\documentclass[12pt,a4paper]{article}

\usepackage{comment}
\usepackage{mathtools}
\usepackage{amsmath}
\usepackage{amsthm}
\usepackage{amsfonts}
\usepackage{amssymb}
\usepackage{latexsym }
\usepackage{float}

\usepackage{tikz-cd}

\usepackage{graphicx,xcolor}

\usepackage{tikz}
\usepackage[framemethod=tikz]{mdframed}
\usepackage[abs]{overpic}

\theoremstyle{plain}
\newtheorem{theo}{Theorem}[section]
\newtheorem{prop}{Proposition}[section]
\newtheorem{cor}{Corollary}[section]
\newtheorem{lemm}{Lemma}[section]
\newtheorem{fact}{Fact}[section]

\theoremstyle{definition}
\newtheorem{defi}{Definition}[section]
\newtheorem*{eg*}{Example}

\theoremstyle{remark}
\newtheorem*{re*}{remark}

\usepackage{amsmath, amssymb, amsthm}

\newcommand{\soejitilde}[6]{\rlap{\hspace{#1 em}\raisebox{#2 ex}{$\widetilde{\phantom{#5}}\hspace{#3 em}\raisebox{#4 ex}{\scriptsize #6}$}}{#5}}
\newcommand{\Soejitilde}[6]{\rlap{\hspace{#1 em}\raisebox{#2 ex}{$\widetilde{\phantom{#5}}\hspace{#3 em}\raisebox{#4 ex}{\scriptsize $#6$}$}}{#5}}

\makeatletter
\newcommand{\biggg}{\bBigg@{7}}
\makeatother

\begin{document}

\title{A topological proof of Terao's generalized Arrow's Impossibility Theorem}
\author{Takuma Okura}
\date{\today}

\maketitle
\tableofcontents

\begin{abstract}
In Terao~\cite{Terao2006ChambersOA}, Hiroaki Terao defined and studied ``admissible map'', which is a generalization of ``social welfare function'' in the context of hyperplane arrangements.
Using this, he proved a generalized Arrow's Impossibility Theorem using combinatorial arguments.
This paper provides another proof of this generalized Arrow's Impossibility Theorem, using the idea of algebraic topology.
\end{abstract}

\section{Introduction}
In the mid-20th century, Kenneth Joseph Arrow defined ``social welfare function'' to discuss the mechanism for determining the preference of an entire society based on individual preferences. He proved that there is no social welfare function satisfying certain conditions. This theorem is known as Arrow's impossibility theorem.
\footnote{Until now, there have been many works related to Arrow's impossibility theorem. Among them, Eliaz's one (Eliaz~\cite{MR2040998}) and works by Gibbert and Satterthwaite (Gibbert~\cite{MR0441407}, Satterthwaite~\cite{MR0414051}) are also approached by topological method (Baryshnikov, Root~\cite{baryshnikov2024topological}, Fujimoto~\cite{Fuji}, Tanaka~\cite{MR2233333})}
(Arrow~\cite{arrow1950difficulty}, \cite{1194c2af-bbd9-34a4-ab4a-a33e7650b716})
While his approach was combinatorial, Graciela Chichilnisky introduced a topological point of view to social choice theory.
(Chichilnisky~\cite{chichilnisky1976manifolds}$\sim$\cite{MR0720115})
Her approach was topological in the formulation of the problem, making it distinct from the traditional formulation of Arrow.
In (Baryshnikov~\cite{BARYSHNIKOV1993404}), Yuliy Baryshnikov opened up a new possibility of using geometry solely in the method of proof, while keeping the content of the traditional Arrow's impossibility theorem unchanged.
Following this, some topological proofs of Arrow's Impossibility theorem are considered. (Manabe~\cite{Mana}, Rajsbaum, Ravent\'{o}s~\cite{rajsbaum2022combinatorial}, Tanaka~\cite{MR2201518}, \cite{MR2220591}, \cite{MR2493199})

On the other hand, in Terao~\cite{Terao2006ChambersOA}, Terao Hiroaki presented a new theorem (hereafter referred to as the ``main theorem'' in this paper) which properly contains the original Arrow's impossibility theorem.
He achieved this by using tools of hyperplane arrangements.
\footnote{As another generalization of Arrow's impossibility theorem, the following preprint was recently published. (Lara, Rajsbaum, Ravent\'{o}s~\cite{lara2024generalization}) It provides a generalization in the context of domain restriction.}
However, his proof method was combinatorial.
The question of this paper is whether the topological method can be applied to prove Terao's generalized Arrow impossibility theorem. In this paper, we show the answer is ``Yes'' by giving a new topological proof of Terao's theorem (chapter \ref{Proof}).

The geometric proof consists of two parts.
First, I constructed simplicial complexes and simplicial maps from combinatorial objects and focused on their homological properties, based on  Manabe~\cite{Mana} and Baryshnikov~\cite{BARYSHNIKOV1993404}.
Following Baryshnikov~\cite{BARYSHNIKOV1993404}, I used nerve and nerve theorem to this end.
Second, following Manabe~\cite{Mana}, I showed the existence of a dictator using metric.
In Manabe~\cite{Mana}, he used a combinatorial argument to reduce Arrow's impossibility theorem to the case of three alternatives. However, this paper does not employ such a reduction. I give a proof that is valid in the general case.
In Baryshnikov~\cite{BARYSHNIKOV1993404}, he overcame the second step by handling concrete cycles in the homology group. While it may be possible to extend this method to Terao's framework, that is not attempted in this study.\\



\section{Original Arrow's Impossibility Theorem}
Arrow's impossibility theorem is a theorem that states no voting system satisfies certain conditions. In this context, the voting system is formulated as a \textit{social welfare function} defined as follows.
Let us think of a society that consists of $m$ people. Suppose each individual has \textit{preferences}, that is linear ordering, over the elements of a certain finite set $A$.
When we denote the set of all total orders on $A$ as $L_{A}$, the \textit{social welfare function} is defined as a map from $L_{A}^{m}$ to $L_{A}$.
However, when considering all possible functions, democracy is not reflected at all. Therefore, it is meaningful to impose some constraints.
He imposed the following three conditions for social welfare function to be reasonable.\\

The first one is the following.
 \begin{defi}[Independence of Irrelevant Alternatives(IIA)]\label{IIA}
Let $A$ be an arbitrary finite set and $m, l$ be natural numbers. 
We denote the set of all total orders on $A$ as $L_{A}$.
We say a map $f\colon L_{A}^{m} \to L_{A}^{l}$ satisfies \textit{Independence of Irrelevant Alternatives $($IIA$)$} when for all distinct two elements $a$, $b$ in $A$, there exists a map $\phi_{\{a, b\}}\colon L_{\{a, b\}}^{m} \to L_{\{a, b\}}^{l}$ which makes the following diagram commutative. Here, $\epsilon_{\{a, b\}} \colon L_A \to L_{\{a, b\}}$ denotes a map that restricts relations on A to that on subset $\{a, b\}$.

\begin{figure}[ht]

\centering


\includegraphics{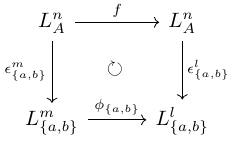}

\end{figure}

\end{defi}

This condition states that the ranking of $a$ and $b$ in the preferences of the entire society is determined only by the information on the ranking of $a$ and $b$ in the preferences of each individual.\\

The second condition is the \textit{Pareto Property $($PAR$)$}.
\begin{defi}[Pareto Property (PAR)]\label{PAR}
Let $A$ be an arbitrary finite set and $m$ be a natural number. 
We denote the set of all total orders on $A$ as $L_{A}$.
We say a map $f\colon L_{A}^{m} \to L_{A}$ satisfies \textit{Pareto Property $($PAR$)$} when $f\circ\Delta = id$.
Here, $\Delta\colon L_{A} \to L_{A}^{m}$ is the diagonal map. 
\end{defi}

This condition states that when all members of the society have the same preference $p_0$, the preference of the society should also be $p_0$.\\

The last condition is \textit{Nondictatorship $($ND$)$}
\begin{defi}[Dictator, Nondictatorship (ND)]\label{ND}
Let $A$ be an arbitrary finite set and $m$ be a natural number. 
We denote the set of all total orders on $A$ as $L_{A}$.
For a map $f\colon L_{A}^{n} \to L_{A}$, $i$ $(1\leq i \leq n)$ is said to be a \textit{dictator} when $f$ coincides with the projection onto the \textit{i}-th component.
We say a map $f\colon L_{A}^{n} \to L_{A}$ satisfies \textit{Nondictatorship $($ND$)$} if there is no dictator for $f$.
\end{defi}

This condition excludes social welfare functions that always follow the opinion of a particular individual.\\

Surprisingly, Arrow's impossibility theorem states that there is no social welfare function that satisfies these three conditions. 

\begin{prop}[Arrow's impossibility theorem]\label{Arrow}
Let $A$ be a finite set which satisfies $|A| \geq 3$.
Let $m$ be natural numbers. 
We denote the set of all total orders on $A$ as $L_{A}$.
A map $f\colon L_{A}^{n} \to L_{A}$ which satisfies IIA and PAR is a projection onto a certain component.
\end{prop}

\section{Formulation by Terao}
In this section, we state Terao's extended Arrow's impossibility theorem based on Terao~\cite{Terao2006ChambersOA}.\\

First, we review basic concepts that concern \textit{hyperplane arrangement}.
\begin{defi}[hyperplane arrangement]
Let $V$ be a vector space.
Finite set of affine subspace of $V$ is said to be \textit{hyperplane arrangement} in $V$.
\end{defi}

\begin{eg*}[Boolean Arrangement]
Let \textit{n} be a natural number.\\
$\mathcal{A}=\{\textrm{ker}x_1, \textrm{ker}x_2, ... ,\textrm{ker}x_n\}$ defines a hyperplane arrangement in $\mathbb{R}^n$.
This is called \textit{Boolean Arrangement}.
\end{eg*}

\begin{eg*}[Braid Arrangement]
Let $n$ be a natural number.
$\mathcal{A}=\{\textrm{ker}(x_i-x_j) \mid 1\leq i < j \leq n, i , j \in \mathbb{Z} \}$ defines a hyperplane arrangement in $\mathbb{R}^n$.
This is called \textit{Braid Arrangement}.
\end{eg*}

\begin{defi}[subarrangement]
Let $\mathcal{A}$ be a hyperplane arrangement in a vector space $V$.
A subset $\mathcal{B}$ of $\mathcal{A}$ is also a hyperplane arrangement in $V$.
So, we say $\mathcal{B}$ is a \textit{subarrangement} of $\mathcal{A}$.
\end{defi}

\begin{defi}[central]
A hyperplane arrangement $\mathcal{A}$ is said to be \textit{central} if all of its element contains origin.
\end{defi}

\begin{defi}[chamber, \textbf{Ch}$(\mathcal{A})$]
Let $\mathcal{A}$ be a hyperplane arrangement in $V$.
A connected component of $V-\underset{H \in \mathcal{A}}{\bigcup} H$ is said to be a \textit{chamber} of $\mathcal{A}$.
A set of all chambers of $\mathcal{A}$ is denoted by \textbf{Ch}$(\mathcal{A})$.
\end{defi}

\begin{defi}[rank, $r(\mathcal{A})$]
Let $\mathcal{A}$ be a central hyperplane arrangement in $V$.
$\textrm{dim}V-\textrm{dim}(\underset{H \in \mathcal{A}}{\bigcap} H)$ is said to be \textit{rank} of $\mathcal{A}$ and denoted by $r(\mathcal{A})$.
\end{defi}

\begin{defi}[decomposable, indecomposable]
Let $\mathcal{A}$ be a central hyperplane arrangement in $V$.
$\mathcal{A}$ is said to be \textit{decomposable} when there are nonempty subarrangements $\mathcal{A}_1$, $\mathcal{A}_2$,\ldots, $\mathcal{A}_n$ ($n \geq 2$)
such that $\mathcal{A} = \mathcal{A}_1 \sqcup \mathcal{A}_2\sqcup,\ldots,\sqcup\mathcal{A}_n$ and $r(\mathcal{A})$ = $r(\mathcal{A}_1) + r(\mathcal{A}_2) + \cdots + r(\mathcal{A}_n)$.
We denote this as $\mathcal{A} = \mathcal{A}_1 \uplus \mathcal{A}_2 \uplus \cdots \uplus \mathcal{A}_n$.
We say $\mathcal{A}$ is \textit{indecomposable} if $\mathcal{A}$ is not decomposable.
\end{defi}

\begin{fact}[Terao~\cite{Terao2006ChambersOA}, Lemma 2.1. and Proposition 2.3.]
Let $\mathcal{A}$ be a central hyperplane arrangement in $V$.
Up to order, there is a unique decomposition $\mathcal{A} = \mathcal{A}_1 \uplus \mathcal{A}_2 \uplus \cdots \uplus \mathcal{A}_n$ such that each $\mathcal{A}_1, \mathcal{A}_2, \ldots,\mathcal{A}_n$ is indecomposable.
This decomposition coincides with the decomposition of the graph $\Gamma(\mathcal{A})$ (definition $\mathrm{\ref{graph}}$) into connected components.
\end{fact}

\begin{eg*}[Boolean Arrangement is decomposable]
Let $n$ be a natural number such that $n \geq 2$.
Let $\mathcal{A}$ be Boolean Arrangement in $\mathbb{R}^n$.
If we set $\mathcal{A}_1=\{\textrm{ker}x_1\}$ and $\mathcal{A}_2=\{\textrm{ker}x_2, \ldots ,\textrm{ker}x_n\}$, then $\mathcal{A} = \mathcal{A}_1 \uplus \mathcal{A}_2$.
Therefore $\mathcal{A}$ is decomposable.
\end{eg*}

\begin{defi}[dependent, independent]
Let $\mathcal{A}$ be a real central hyperplane arrangement in $V$.
We say subarrangement $\mathcal{B}$ is \textit{dependent} when $r(\mathcal{B}) < |\mathcal{B}|$.
We say $\mathcal{B}$ is \textit{independent} when it is not dependent.
\end{defi}

\begin{defi}[circuit]
Let $\mathcal{A}$ be a  real central hyperplane arrangement in $V$.
We say subarrangement $\mathcal{B}$ is a \textit{circuit} if it is minimal in the set of dependent subarrangements.
\end{defi}

\begin{defi}[graph $\Gamma(\mathcal{A})$]\label{graph}
Let $\mathcal{A}$ be a real central hyperplane arrangement in $V$.
We can construct a graph $\mathit{\Gamma(\mathcal{A}})$ in the following way.
The vertex set is \textbf{Ch}$(\mathcal{A})$. We connect $H_i$ and $H_j$ with an edge if and only if there is a circuit that contains $H_i$ and $H_j$.
\end{defi}

\begin{prop}[criterion for decomposability, Terao~\cite{Terao2006ChambersOA} Lemma 2.1.]
Let $\mathcal{A}$ be a real central hyperplane arrangement in $V$.
$\mathcal{A}$ is indecomposable if and only if $\Gamma(\mathcal{A})$ is connected.
\end{prop}
\begin{proof}
Terao~\cite{Terao2006ChambersOA} Lemma 2.1. 
\end{proof}

\begin{eg*}[Braid Arrangement is indecomposable]
Let \textit{n} be a natural number such that $n \geq 3$. 
Let $\mathcal{A}$ be Braid Arrangement in $\mathbb{R}^n$. 
$\Gamma(\mathcal{A})$ is connected because each subarrangement $\mathcal{B}$ such that $|\mathcal{B}| = 3$ is a circuit.
Therefore, Braid Arrangement is indecomposable.
\end{eg*}

\begin{re*}
\rm{From now on, unless otherwise stated, the coefficient field of vector space is real, and hyperplane arrangement is always central.}
\end{re*}

Terao~\cite{Terao2006ChambersOA} successfully formulated Arrow's impossibility theorem by introducing the concept of \textit{admissible map}. In this paper, we divide the definition into two parts (IIA and PAR) to make it easier to understand the correspondence with Arrow's classical impossibility theorem.

First, we define a notation for certain maps in order to state the definition of IIA.
 
 \begin{defi}[map $\epsilon_H$]
Let $\mathcal{A}$ be a real central hyperplane arrangement in $V$.
Let $n$ be a natural number. 
For arbitrary element $H$ in $\mathcal{A}$, we define a map $\epsilon_H \colon \textbf{Ch}(\mathcal{A}) \to \textbf{Ch}(\{H\})$ as follows.
For an element $c$ in $\textbf{Ch}(\mathcal{A})$, $\epsilon_H(c)$ is an element of $\textbf{Ch}(\{H\})$ which contains $c$.
\end{defi}

\begin{defi}[IIA (in hyperplane arrangement setting)]\label{IIA2}
Let $\mathcal{A}$ be a real central hyperplane arrangement in $V$.
Let $m, l$ be natural numbers. 
A map $\Phi \colon \textbf{Ch}(\mathcal{A})^m \to \textbf{Ch}(\mathcal{A})^l$ is said to satisfy \textit{IIA} when for each $H$ there exists a map $\phi_H$
which makes the following diagram commutative.

\begin{figure}[ht]

\centering

\includegraphics{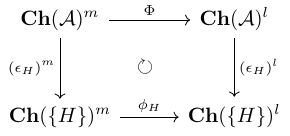}
\end{figure}

\end{defi}

\begin{defi}[PAR (in hyperplane arrangement setting)]
Let $\mathcal{A}$ be a real central hyperplane arrangement in $V$.
Let $m$ be natural numbers. 
We say a map $\Phi \colon \textbf{Ch}(\mathcal{A})^m \to \textbf{Ch}(\mathcal{A})$ satisfies \textit{PAR} when for every chamber $c$ of $\mathcal{A}$, $\Phi(c, c,\ldots, c) = c$ holds.
\end{defi}

\begin{defi}[admissible map, Terao~\cite{Terao2006ChambersOA} Definition 1.1.]
Let $\mathcal{A}$ be a real central hyperplane arrangement in $V$.
Let $m$ be natural numbers. 
A map $\Phi \colon \textbf{Ch}(\mathcal{A})^m \to \textbf{Ch}(\mathcal{A})$ is said to be \textit{admissible} if $\Phi$ satisfies IIA and PAR(in in hyperplane arrangement setting)
\end{defi}

Based on the above definitions, Terao's extended Arrow's impossibility theorem can be stated as follows.
The original Arrow's impossibility theorem corresponds to the case where $\mathcal{A}$ is a Braid Arrangement.
\footnote{More precisely, Terao has classified admissible arrangements, including this theorem as a result. For specific details, please refer to reference Terao~\cite{Terao2006ChambersOA}.}

\begin{theo}[Terao's extended Arrow's impossibility theorem, Terao~\cite{Terao2006ChambersOA} Theorem 1.5. (2)]\label{mainthm}
Let $\mathcal{A}$ be a real central hyperplane arrangement in $V$.
Let $m$ be natural numbers. 
if $|\mathcal{A}|\ge3$ , then every admissible map $\Phi : \textbf{Ch}(\mathcal{A})^m \to \textbf{Ch}(\mathcal{A})$ is a projection onto a component.
\end{theo}

The purpose of this paper is to give a topological proof of this theorem.

\section{Nerve Theorem}
In this section, we introduce a theorem(Proposition \ref{nerve}) from Bauer, Kerber, Roll, and Rolle \cite{BAUER2023125503}. This theorem will be used in the proof in section \ref{Proof} of this paper.

\begin{defi}[category $\textbf{ClConv}_\textbf{*}$]

We define a category $\textbf{\textit{ClConv}}_\textbf{*}$ as follows.

Object : $(X, \{C_i\}_{i=1}^k, \{b_{i_1i_2\ldots i_r}\})$\\
Here, $X$ is a subset of Euclidean space of arbitrary dimension $n$, and $k$ is a natural number.
$\{C_i\}_{i=1}^k$ is a closed convex finite covering of $\{C_i\}_{i=1}^k$ in $\mathbb{R}^n$ (so we have $\bigcup_{i=1}^{k}C_{i} = X$ in $\mathbb{R}^n$).
Furthermore $\{b_{i_1i_2\ldots i_r}\}$ denotes a point in $\bigcap_{j=1}^{r}C_{i_j}$ chosen for each subset $\{i_1,i_2,\ldots,i_r\} \subseteq \{1,2,\ldots,k\}$ such that $\bigcap_{j=1}^{r}C_{i_j} \neq \emptyset$.

Morphism : Morphism from $(X, \{C_i\}_{i=1}^k, \{b_{i_1i_2\ldots i_r}\})$ to $(Y, \{D_j\}_{j=1}^l, \{c_{j_1j_2\ldots j_s}\})$ is a pair $(f,\phi)$.
Here, $f$ is a continuous map $f \colon X \to Y$ which is affine on each $\{C_i\}$, and $\phi$ is a index map $\phi \colon \{1,2,\dots,k\} \to \{1,2,\ldots,l\}$ which satisfies $f(C_i) \subset D_{\phi(i)}$, $f(b_{i_1i_2\dots i_r}) = c_{\phi(i_1)\phi(i_2)\dots\phi(i_r)}$.

The composition of morphisms and other structures are naturally defined, forming a category.
\end{defi}

\begin{defi}[Forgetful Functor \textbf{Forget}]

We can define a functor 
$\textbf{\textit{Forget}} \colon \textbf{ClConv}_\textbf{*} \to \textbf{Top}$ as follows.\\
$\textbf{Forget}((X, \{C_i\}_{i=1}^k, \{b_{i_1i_2\ldots i_r}\})) = X$.\\
$\textbf{Forget}((f,\phi)) = f$.
\end{defi}

\begin{defi}[Functor \textbf{SdNrv}]
We can define a functor 
$\textbf{\textit{SdNrv}}:\textbf{ClConv}_\textbf{*} \to \textbf{Top}$ as follows.\\
$\textbf{SdNrv}((X, \{C_i\}_{i=1}^k, \{b_{i_1i_2\ldots i_r}\}))=|\textrm{SdNrv}(\{C_i\}_{i=1}^k)|$.\\
$\textbf{SdNrv}((f,\phi))=|\textrm{Sd}(\phi_{*})|$.\\
Here, $\phi_{*} \colon \textrm{Nrv}(\{C_i\}_{i=1}^k). \to \textrm{Nrv}(\{D_j\}_{j=1}^l)$ is a map which is induced from $\phi$, and $\textrm{Sd}(\phi_{*})$ is a map induced from $\phi_{*}$ between barycentric subdivisions.
And, $|\textrm{Sd}(\phi_{*})|$ is a map induced from $\textrm{Sd}(\phi_{*})$ between geometric realizations.
\end{defi}

\begin{defi}[Natural transformation \textbf{$\Gamma$}]\label{gamma}
We can define a natural transformation \textbf{$\mathit{\Gamma}$} from $\textbf{SdNrv}$ to $\textbf{Forget}$ as follows.
For an object $(X, \{C_i\}_{i=1}^k, \{b_{i_1i_2\ldots i_r}\})$ in $\textbf{ClConv}_\textbf{*}$, we associate a map from $|\textrm{SdNrv}(\{C_i\}_{i=1}^k)|$ to $X$ by  mapping vertex $\{C_{i_1},C_{i_2},\ldots,C_{i_r}\}$ in $|\textrm{SdNrv}(\{C_i\}_{i=1}^k)|$ to $\{b_{i_1i_2\ldots i_r}\}$ and extending affinely.
\end{defi}

\begin{prop}[\textbf{$\Gamma$} is object-wise homotopy equivalence, Bauer, Kerber, Roll, Rolle~\cite{BAUER2023125503} Theorem 3.1.]\label{nerve}
Every morphism in \textbf{Top} induced from the natural transformation defined above is homotopy equivalence.
\end{prop}
\begin{proof}
Bauer, Kerber, Roll, Rolle~\cite{BAUER2023125503} Theorem 3.1.
\end{proof}

\section{Simplicial complex based on Manabe}\label{Manabe}

In this section, we construct a simplicial complex $M_m(\mathcal{A})$ from a given hyperplane arrangement $\mathcal{A}$ and a natural number $m$ (Definition \ref{ManabeSimp}).
This is inspired by the simplicial complex $K$ described in Manabe \cite{Mana}.

\begin{defi}[separate]
Let $\mathcal{A}$ be a real central hyperplane arrangement in $V$.\\
Let $m$ be a natural number.\\
For an element $H$ in $\mathcal{A}$ and a subset $T \subseteq \textbf{Ch}(\mathcal{A})^m$, we say $\textbf{Ch}(\mathcal{A})$ \textit{separate} $T$ when $|(\epsilon_H)^m(T)| \geq 2$.
\end{defi}

\begin{defi}[$M_m(\mathcal{A})$]\label{ManabeSimp}
Let $\mathcal{A}$ be a real central hyperplane arrangement in $V$.\\
Let $m$ be a natural number.\\
We construct a simplicial complex $\mathit{M_m(\mathcal{A})}$ as follows.

Vertex set : $\textbf{Ch}(\mathcal{A})^m$\\

The condition for a simplex : Subset $T \subseteq M_m(\mathcal{A})$ is an element of  $M_m(\mathcal{A})$ if and only if there is at least one element of $\mathcal{A}$ which does not separate $T$.
Since this condition is preserved by the operation of taking subsets, a face of a simplex is still a simplex.
Therefore this satisfies the axioms of an abstract simplicial complex.
\end{defi}

\begin{prop}[IIA map induces a map between simplicial complex]
Let $\mathcal{A}$ be a real central hyperplane arrangement in $V$.\\
Let $m$, $l$ be a natural number.\\
If $f \colon \textbf{Ch}(\mathcal{A})^m \to \textbf{Ch}(\mathcal{A})^l$ satisfies IIA, $f$ induces a simplicial map $\soejitilde{0}{0.5}{-0.1}{1.2}{f}{M}\colon M_m(\mathcal{A}) \to M_{l}(\mathcal{A})$.
\end{prop}
\begin{proof}
This is clear from definition.
\end{proof}

\section{Simplidial complex based on Baryshnikov}\label{Baryshnikov}

In this section we construct a simplicial complex $B_m(\mathcal{A})$ from a given hyperplane arrangement $\mathcal{A}$ and a natural number $m$ (Definition \ref{BaryshnikovSimp}).
This is a complete generalization of simplicial complexes ${N_W}$ and ${N_P}$ in Baryshnikov~\cite{BARYSHNIKOV1993404}.

\begin{defi}[$\mathfrak{U}_{\mathcal{A},m}$, $U_{\mathcal{A},m}$]
Let $\mathcal{A}$ be a real central hyperplane arrangement in $V$.\\
Let \textit{m} be a natural number.\\
We define $\mathfrak{U}_{\mathcal{A},m}$ as  $\{c_1 \times c_2 \times \cdots \times c_m \subset V^m |c_1,c_2,\ldots,c_m\in \textbf{Ch}(\{H\}),H\in\mathcal{A}\} $ and define $U_m$ as the union of all elements of $\mathfrak{U}_{\mathcal{A},m}$.
$U_m$ is an open subset of $V^m$.
\end{defi}

\begin{defi}[$\mathfrak{A}_{\mathcal{A},m}$, $A_{\mathcal{A},m}$]\label{A}
\footnote{We introduce this concept for just a technical reason. In the proof of our main theorem, we want to use functorial Nerve theorem. We use this definition and use a ``closed'' covering to ensure the functoriality. There might be a simpler method, but considering that the technical difficulties here are not essential to the purpose of this paper, we have not explored that possibility.}
Let $\mathcal{A}$ be a real central hyperplane arrangement in $V$.\\
For each element $H$ of $\mathcal{A}$, we fix an element $\alpha_{H}$ of $V^{*}$ which satisfies ker($\alpha_{H}$) = $H$.
Let $m$ be a natural number.\\
We define $\mathfrak{A}_{\mathcal{A},m}$ as $\{a_1 \times a_2 \times \cdots \times a_m \subset V^m |a_1,a_2,\ldots,a_m\in \textbf{A}(\{H\}),H\in\mathcal{A}\}$ and define $A_m$ as the union of all elements of $\mathfrak{A}_{\mathcal{A},m}$.
Here, $\textbf{A}(\{H\}) = \{ c_0\cap\{\vec{x} \in V | |\alpha_{H}(\vec{x})| \geq 1\},  c_1\cap\{\vec{x} \in V | |\alpha_{H}(\vec{x})| \geq 1\}\}$ ($c_{0}$ and $c_{1}$ are two chambers of $\{H\}$)
$A_m$ is a closed subset of $V^m$.
\end{defi}

\begin{defi}[$B_m(\mathcal{A})$]\label{BaryshnikovSimp}
Let $\mathcal{A}$ be a real central hyperplane arrangement in $V$.\\
Let $m$ be a natural number.\\
Since $\mathfrak{U}_{\mathcal{A},m}$ is a covering of $U_m$, we can construct a simplicial complex $B_m(\mathcal{A})$ as a nerve of this covering $\textrm{Nrv}(\mathfrak{U}_{\mathcal{A},m}) (= \textrm{Nrv}(\mathfrak{A}_{\mathcal{A},m})\footnote{By definition, we can use $\mathfrak{A}_{\mathcal{A},m}$ as a covering of $A_m$ instead of $\mathfrak{U}_{\mathcal{A},m}$ in order to get $B_m(\mathcal{A})$.})$.
\end{defi}

\begin{prop}[structure of $U_{\mathcal{A},m}$]\label{str}
Let $\mathcal{A}$ be a real central hyperplane arrangement in $V$.\\
Let $m$ be a natural number.\\
We have the following.\\

\noindent
$(1)$ $U_m = V^m - \underset{(i_1,i_2,\ldots,i_n)\in\{1,2,\ldots,m\}^n}{\bigcup}(\pi^{-1}_{i_1}(H_1)\cap\pi^{-1}_{i_2}(H_2)\cap \ldots \cap\pi^{-1}_{i_n}(H_n))$ 

\noindent
$(2)$ $\textrm{codim}(\pi^{-1}_{i_1}(H_1)\cap\pi^{-1}_{i_2}(H_2)\cap \ldots \cap\pi^{-1}_{i_n}(H_n)) \geq r(\mathcal{A})$

\noindent
$(3)$ The equation holds in $(\mathrm{2})$ 
$\iff$ $s \colon \mathcal{A} \to \{1,2,\ldots,m\} \,\,(s(H_k) = i_k)$ is constant on each connected component of $\Gamma(\mathcal{A})$.
\end{prop}

\begin{proof}
\noindent
(1)
\begin{align*}
U_m &= \underset{U\in\mathfrak{U}_{\mathcal{A},m}}{\bigcup{U}}\\ 
&=\{(\vec{x_1},\vec{x_2},\ldots,\vec{x_m})\in V^m|\exists H \in \mathcal{A}, \forall i \in \{1,2,\ldots,m\}, \vec{x_i} \notin H \}\\
&= V^m-\{(\vec{x_1},\vec{x_2},\ldots,\vec{x_m})\in V^m|\forall H \in \mathcal{A}, \exists i \in \{1,2,\ldots,m\}, \vec{x_i} \in H \}\\
&=V^m-{\bigcup_{\substack{(i_1,i_2,\ldots,i_n)\in\{1,2,\ldots,m\}^n}}}{(\pi_{i_1}^{-1}(H_1)\cap\pi_{i_2}^{-1}(H_2)\cap \ldots \cap\pi_{i_n}^{-1}(H_n))}
\end{align*}

Here, $\mathcal{A} = \{H_1, H_2,\ldots,H_n\}$.\\

(2)(3)\\

We have 
\begin{align*}
\pi_{i_1}^{-1}(H_1)\cap\pi_{i_2}^{-1}(H_2)\cap \ldots \cap\pi_{i_n}^{-1}(H_n) &= \bigcap_{j=1}^{m}\pi_{j}^{-1}(\bigcap_{H\in s^{-1}(j)}H)
\end{align*}

(If $s^{-1}(j) = \emptyset$,we assume $\bigcap_{H\in s^{-1}(j)}H = V$)\\

Therefore

\begin{align*}
\textrm{codim}_{V^m}(\bigcap_{j=1}^{m}\pi_{j}^{-1}(\bigcap_{H\in s^{-1}(j)}H))
&=\sum_{j=1}^{m}\textrm{codim}_{V}(\bigcap_{H\in s^{-1}(j)}H)\\
&\geq \textrm{codim}_{V}(\bigcap_{j=1}^{m}\bigcap_{H\in s^{-1}(j)}H)\\
&=r(\mathcal{A})\\
\end{align*}

This equality holds if and only if $\mathcal{A} = \biguplus_{j=1}^{m}s^{-1}(j)$. This is equivalent to the condition that $s \colon \mathcal{A} \to \{1,2,\ldots,m\} \,\,(s(H_k) = i_k)$ is constant on each connected component of $\Gamma(\mathcal{A})$.

\end{proof}


\section{Relation between $M_m(\mathcal{A})$ and $B_m(\mathcal{A})$}

\begin{defi}[dual of simplicial complex]
Let $K$ be a simplicial complex such that the union of all simplicies with maximum dimension is $K$.
In other words, the set of all simplicies with maximum dimension becomes a covering of $K$.
In this case, we define the \textit{dual} of $K$, $K^*$ as the nerve of this covering.
\end{defi}

\begin{prop}[$M_m(\mathcal{A})$ can be identified with the dual of $B_m(\mathcal{A})$]
Let $\mathcal{A}$ be a real central hyperplane arrangement in $V$.
Let $m$ be a natural number.
$B_m(\mathcal{A})$ is a simplicial complex such that the set of all simplicies with maximum dimension becomes its covering.
We have a natural one-to-one correspondence of elements between $B_m$ $B_m(\mathcal{A})$ and $\textbf{Ch}(\mathcal{A})^m$, and we can think of  $M_m(\mathcal{A})$ as the dual of $B_m(\mathcal{A})$.
\end{prop}
\begin{proof}
This is clear from definition.
\end{proof}

\section{Topological proof of Terao's extended Arrow's impossibility theorem}\label{Proof}

Let $\mathcal{A} = \{H_1, H_2, \ldots, H_n\}$ be a central hyperplane arrangement in a real vector space $V$, and let $n =|\mathcal{A}|\ge3$.
Since we use $A_{\mathcal{A},m}$ in the following discussion, we fix an element of $V^{*}$ for each element $H$ in $\mathcal{A}$.
Let $m$ be a natural number, and $\Phi \colon \textbf{Ch}(\mathcal{A})^m \to \textbf{Ch}(\mathcal{A})$ be an admissible map.
We want to show that $\Phi$ is a projection map onto a certain component.

First, we relate $M_m(\mathcal{A})$ and $B_m(\mathcal{A})$ to the sphere in a topological point of view.
We use the nerve theorem for that purpose.
The first key step is to prove Proposition \ref{1}. 
This part is heavily informed by Baryshnikov~\cite{BARYSHNIKOV1993404}. 

\subsection{A commutative diagram}
We can construct the following two commutative diagrams (Figure \ref{kakanzusiki_1}, \ref{kakanzusiki_2}) for $\vec{c} = (c_1,c_2,\ldots,c_{m-1}) \in \textbf{Ch}(\mathcal{A})^{m-1}$.
These commutative diagrams are largely the same, except for the mappings indicated by the dotted lines.\\
These diagrams relate simplicial complexes $M_{\bullet}(\mathcal{A})$ and $B_{\bullet}(\mathcal{A})$ which are constructed from combinatorial information of Arrow's impossibility theorem 
to spheres functorially in a topological point of view. In this sense, this diagram plays the role of a bridge between combinatorics and topology.
\begin{figure}[H]

\includegraphics{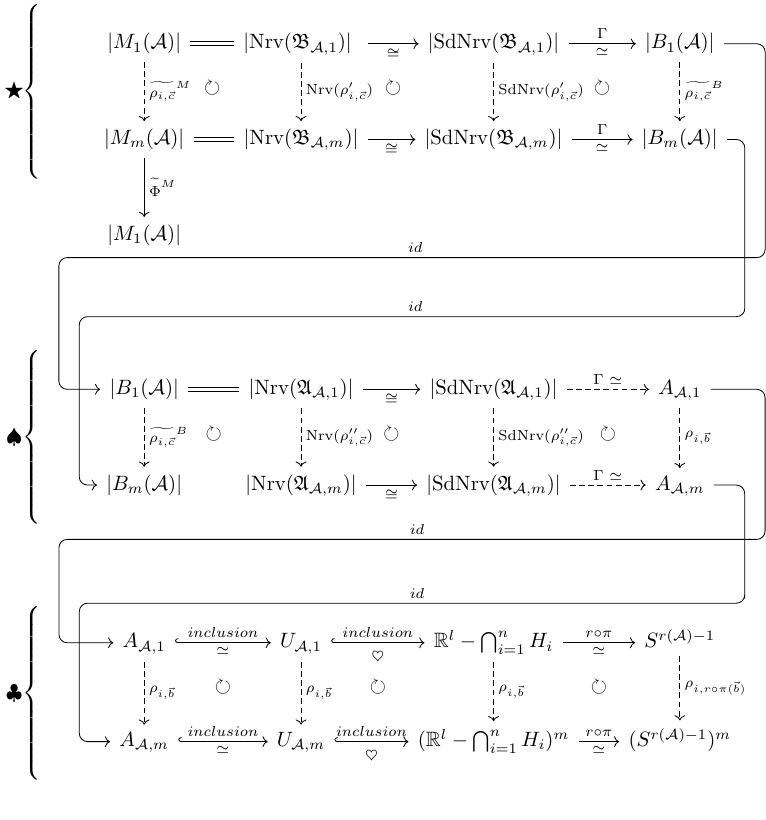}

\caption{a commutative diagram}\label{kakanzusiki_1}

\end{figure}
\begin{figure}[H]

\includegraphics{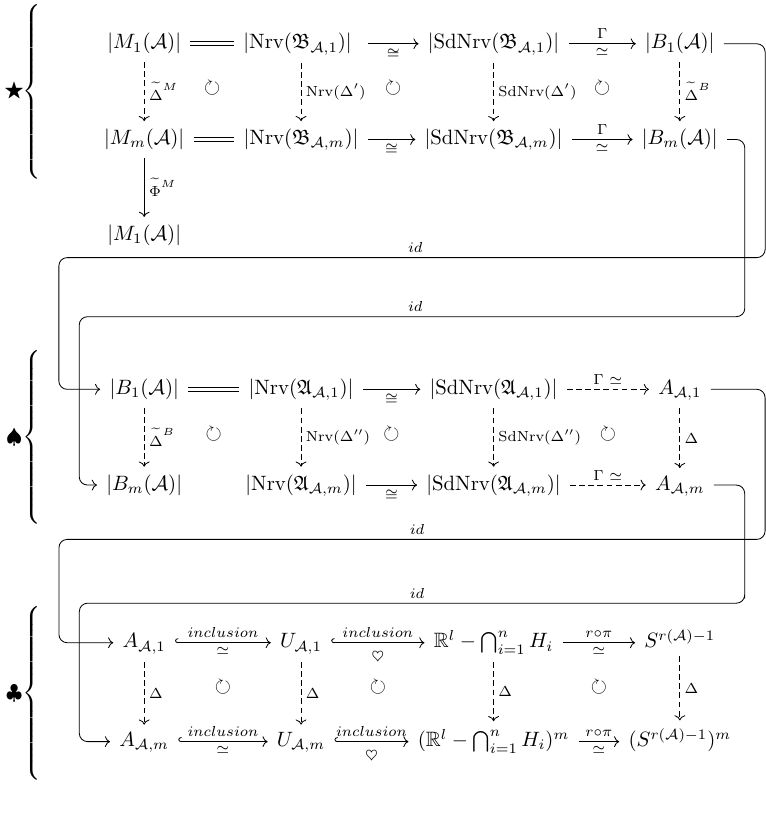}

\caption{a commutative diagram}\label{kakanzusiki_2}

\end{figure}

We explain some notations in Figure \ref{kakanzusiki_1}, \ref{kakanzusiki_2}.
\begin{itemize}
\item$\rho_{i,\vec{c}}$ and $\Delta$ are maps defined as follows.

\[
    \begin{array}{r@{\,\,}c@{\,\,}c@{\,\,}c}
       \rho_{i,\vec{c}}\colon&\textbf{Ch}(\mathcal{A})&\xrightarrow{\qquad\qquad\qquad}&\textbf{Ch}(\mathcal{A})^m\\
       &\rotatebox{90}{$\in$}&&\rotatebox{90}{$\in$}\\
       &c&\xmapsto{\qquad\qquad\qquad}&(c_1,c_2,...,c_{i-1},c,c_{i},\ldots,c_{m-1})
    \end{array}
\]

\[
\hspace{-2.7cm}
    \begin{array}{r@{\,\,}c@{\,\,}c@{\,\,}c}
       \Delta\colon&\textbf{Ch}(\mathcal{A})&\xrightarrow{\qquad\qquad\qquad}&\textbf{Ch}(\mathcal{A})^m\\
       &\rotatebox{90}{$\in$}&&\rotatebox{90}{$\in$}\\
       &c&\xmapsto{\qquad\qquad\qquad}&(c,c,\ldots,c)
    \end{array}
\]

Since it is clear from definition that $\rho_{i,\vec{c}}$ and $\Delta$ satisfy IIA, each of them is uniquely extended to a simplicial map from $|M_1(\mathcal{A})|$ to $|M_m(\mathcal{A})|$.
We denote them as 
$\Soejitilde{0}{0.5}{-0.2}{0.8}{\rho_{i,\vec{c}}}{M}$
 and 
$\Soejitilde{0}{0}{-0.2}{1.2}{\Delta}{M}$.\\
Similarly $\Phi$ can be uniquely extended to a simplicial map from $|M_1(\mathcal{A})|$ to $|M_m(\mathcal{A})|$. We denote it as 
$\Soejitilde{0}{0}{-0.2}{1.2}{\Phi}{M}$

\item$\mathfrak{B}_{\mathcal{A},k}$ is the set of maximum dimensional simplicies of $B_k(\mathcal{A})$.

\item We define simplicial maps
$\Soejitilde{0}{0.5}{-0.2}{0.8}{\rho_{i,\vec{c}}}{M}$
 and 
 $\Soejitilde{0}{0}{-0.2}{1.2}{\Delta}{M}$ as follows.\\
As defined earlier, $B_1(\mathcal{A})$ is the nerve of $\mathfrak{A}_{\mathcal{A},1}$ (a covering of $A_1$), and $B_m(\mathcal{A})$ is the nerve of $\mathfrak{A}_{\mathcal{A},m}$ (a covering of $A_m$).\\
Let $a$ be an element of $\mathfrak{A}_{\mathcal{A},1}$. We denote a map which sends $a$ to $a_1\times a_2\times \cdots \times a_{i-1}\times a\times a_i\times \cdots \times a_{m-1}$ as $\rho_{i,\vec{c}}^{\prime\prime} \colon \mathfrak{A}_{\mathcal{A},1} \to \mathfrak{A}_{\mathcal{A},m}$.
Here, each $a_k (1\leq k)$ is a unique element of $\mathfrak{A}_{\mathcal{A},1}$ which is contained in $\epsilon_{H}(c_{k})$.
Here, $c_k (1\leq k)$ are chanbers of $\mathcal{A}$ fixed above in order to create this diagram.
Furthermore, $\Delta^{\prime\prime} \colon \mathfrak{A}_{\mathcal{A},1} \to \mathfrak{A}_{\mathcal{A},m}$ is a map which sends $c$ to $c\times c\times \cdots \times c$.\\
Since for some elements of $\mathfrak{A}_{\mathcal{A},1}$, $d_1,\ldots,d_k$, both $\rho_{i,\vec{c}}^{\prime\prime}$ and $\Delta^{\prime\prime}$ satisfy $\bigcap_{i=1}^{k}d_{i} \neq \emptyset \Rightarrow \bigcap_{i=1}^{k}\rho_{i,\vec{c}}^{\prime\prime}(d_{i}) \neq \emptyset$ , $\bigcap_{i=1}^{k}d_{i} \neq \emptyset \Rightarrow \bigcap_{i=1}^{k}\Delta^{\prime\prime}(d_{i}) \neq \emptyset$, each of them can be uniquely extended to a simplicial map from $|B_1(\mathcal{A})|$ to $|B_m(\mathcal{A})|$. We denote them as 
$\Soejitilde{0}{0.5}{-0.2}{0.8}{\rho_{i,\vec{c}}}{B}$
 and 
 $\Soejitilde{0}{0}{-0.2}{1.2}{\Delta}{B}$.

\item
$\rho_{i,\vec{c}}^{\prime}$ and $\Delta^{\prime}$ are maps from $\{|\sigma|\}_{\sigma \in \mathfrak{B}_{\mathcal{A},1}}$ to $\{|\sigma|\}_{\sigma \in \mathfrak{B}_{\mathcal{A},m}}$ induced from $\Soejitilde{0}{0.5}{-0.2}{0.8}{\rho_{i,\vec{c}}}{B}$ and  $\Soejitilde{0}{0}{-0.2}{1.2}{\Delta}{B}$.

\item
$| \cdot |$ represents geometric realization of abstract simplicial complex.

\item
Let $\mathfrak{U}$ be some covering. $\textrm{Nrv}(\mathfrak{U})$ is the nerve of $\mathfrak{U}$. This is an abstract simplicial complex whose vertex set is $\mathfrak{U}$.
$\textrm{SdNrv}(\mathfrak{U})$ is the barysentric subdivision of $\textrm{Nrv}(\mathfrak{U})$. This is an abstract simplicial complex whose vertex is simplex of $\textrm{Nrv}(\mathfrak{U})$ and whose simplex is ascending chain of simplexes of $\textrm{Nrv}(\mathfrak{U})$. When we have a map between covering $f \colon \mathfrak{U} \to \mathfrak{V}$.
$\textrm{Nrv}(f) \colon |\textrm{Nrv}(\mathfrak{U})| \to |\textrm{Nrv}(\mathfrak{V})|$ is the affine expansion of $f$. We can also induce $\textrm{SdNrv}(f) \colon |\textrm{SdNrv}(\mathfrak{U})| \to |\textrm{SdNrv}(\mathfrak{V})|$ from $f$.

\item
$\vec{b}$ is introduced in subsection \ref{2dannme}

\item $\pi \colon \mathbb{R}^l \to (\bigcap_{i=1}^{n}H_{i})^{\perp}$ is the orthogonal projection, and $r$ is a deformation retract $r(x) = x/|x|$.

\item Each $\Gamma$ in the diagram above denotes a map that is induced from the natural transformation $\Gamma$ (Proposition \ref{nerve})

\end{itemize}

\subsection{the ($\bigstar$) part in Figure \ref{kakanzusiki_1}, \ref{kakanzusiki_2} (the first nerve theorem)}
Let us consider $(|B_1(\mathcal{A})|,\mathfrak{B}_{\mathcal{A},1},\{b_{\sigma}\}_{\sigma \in \mathfrak{B}_{\mathcal{A},1}})$ and $(|B_m(\mathcal{A})|,\mathfrak{B}_{\mathcal{A},m},\{b_{\sigma}\}_{\sigma \in \mathfrak{B}_{\mathcal{A},m}})$, objects in category $\textbf{ClConv}_\textbf{*}$, and morphisms 
$(\Soejitilde{0}{0.5}{-0.2}{0.8}{\rho_{i,\vec{c}}}{B},  \rho_{i,\vec{c}}^{\prime})$ and $(\Soejitilde{0}{0}{-0.2}{1.2}{\Delta}{B},\Delta^{\prime})$

  between them.
Here, $b_{\sigma}$ denotes the barycenter of $\sigma$.

It is clear that $(|B_1(\mathcal{A})|,\mathfrak{B}_{\mathcal{A},1},\{b_{\sigma}\}_{\sigma \in \mathfrak{B}_{\mathcal{A},1}})$ and $(|B_m(\mathcal{A})|,\mathfrak{B}_{\mathcal{A},m},\{b_{\sigma}\}_{\sigma \in \mathfrak{B}_{\mathcal{A},m}})$ are actually determine objects in category $\textbf{ClConv}_\textbf{*}$, and that
$(\Soejitilde{0}{0.5}{-0.2}{0.8}{\rho_{i,\vec{c}}}{B}, \rho_{i,\vec{c}}^{\prime})$ and $(\Soejitilde{0}{0}{-0.2}{1.2}{\Delta}{B},\Delta^{\prime})$ are actually morphisms between them.

The commutativity of the diagram is evident for the middle and right squares and similarly follows easily for the leftmost square from the definition of the mappings and the manner of identification. 
The mappings induced by the two $\Gamma$s in the ($\bigstar$) part are both homotopy equivalent maps by the nerve theorem (Proposition \ref{nerve}).

\subsection{the ($\spadesuit$) part of Figure \ref{kakanzusiki_1}, \ref{kakanzusiki_2} (the second nerve theorem)}\label{2dannme}
First, we fix a set of base points of $\mathfrak{A}_{\mathcal{A},1}$. We denote this as $B$.
Next, for $c \in \textbf{Ch}(\mathcal{A})$, a set of base points $B_{c}$ for $\mathfrak{A}_{\mathcal{A},1}$ is defined as follows.
If $A_1,\ldots,A_k \in \mathfrak{A}_{\mathcal{A},1}(\bigcap_{i=1}^{k}A_i \neq \emptyset)$ satisfies $c\cap\bigcap_{i=1}^{k}A_i \neq \emptyset$, then the base point for $\{A_1,\ldots,A_k\}$ is defined as $\vec{b_c}$, where $\vec{b_c}$ is the base point of $A_{H,c}$ in $B$. Here, $A_{H,c} = \{A \in \mathfrak{A}_{\mathcal{A},1} | A \cap c \neq \emptyset\}$.
For $A_1,\ldots,A_k \in \mathfrak{A}_{\mathcal{A},1}(\bigcap_{i=1}^{k}A_i \neq \emptyset)$ which do not satisfy the above condition, the base point of $\{A_1,\ldots,A_k\}$ in $B_{c}$ is defined to be the same as that in $B$.

Generally, it should be noted that from the sets of base points $B_1,\ldots,B_m$ for the covering $\mathfrak{A}_{\mathcal{A},1}$ of $A_{\mathcal{A},1}$, a set of base point $B_1\times \cdots \times B_m$ for the covering $\mathfrak{A}_{\mathcal{A},m}$ of $A_{\mathcal{A},m}$ is induced naturally.

Let us consider objects $(A_{\mathcal{A},1},\mathfrak{A}_{\mathcal{A},1},B_c)$, $(A_{\mathcal{A},m},\mathfrak{A}_{\mathcal{A},m},B_{c_1}\times B_{c_2}\times \cdots \times B_{c_{i-1}}\times B_{c}\times B_{c_i}\times \cdots \times B_{c_{m-1}})$, and a morphism $(\rho_{i,\vec{b}},  \rho_{i,\vec{c}}^{\prime\prime})$ between them in category $\textbf{ClConv}_\textbf{*}$.
Furthermore, let us also consider objects $(A_{\mathcal{A},1},\mathfrak{A}_{\mathcal{A},1},B_c)$, $(A_{\mathcal{A},m},\mathfrak{A}_{\mathcal{A},m},B_c \times \cdots \times B_c)$, and a morphism $(\Delta,\Delta^{\prime\prime})$ between them in the same category $\textbf{ClConv}_\textbf{*}$.

Here, we define $\vec{b}$ as $\vec{b} = (\vec{b_{c_1}},\vec{b_{c_2}},\ldots,\vec{b_{c_{m-1}}}) \in \textbf{Ch}(\mathcal{A})^{m-1}$, and $\rho_{i,\vec{b}}$ as $\rho_{i,\vec{b}}(\vec{x})= (\vec{b_{c_1}},\vec{b_{c_2}},\ldots,\vec{b_{c_{i-1}}},\vec{x},\vec{b_{c_i}},\ldots,\vec{b_{c_{m-1}}})$. $\Delta$ is the diagonal map.
The choice of base points ensures that these objects and morphisms are actually objects and morphisms of $\textbf{ClConv}_\textbf{*}$.
The commutativity of the diagram is directly follows from definitions and naturality.
By the nerve theorem (Proposition \ref{nerve}), the maps induced by the four $\Gamma$s in the ($\bigstar$) part are all homotopy equivalence.
(Note that the $\Gamma$s defined in Definition \ref{gamma} are homotopic to each other even when the base points are changed. This property plays a significant role when deducing Proposition \ref{1})

\subsection{the ($\clubsuit$) part of Figure \ref{kakanzusiki_1}, \ref{kakanzusiki_2}}

\begin{lemm}[the leftmost two inclusions are homotopy equivalences]
In ($\clubsuit$) part of the Figure $\ref{kakanzusiki_1}$ and $\ref{kakanzusiki_2}$, the leftmost two inclusions of are homotopy equivalences 
\end{lemm}
\begin{proof}
see section \ref{hosoku1}.
\end{proof}

\begin{lemm}[two inclusion maps ($\heartsuit$) induce  isomorphisms on the homology groups of degree $r(\mathcal{A})-1$ or lower.]
\end{lemm}
\begin{proof}
This follows from Proposition \ref{str}.
\end{proof}

This follows purely from algebraic topology.
\begin{fact}[a proposition  concerning  sphere]\label{sphere}
Let $n, m$ be natural numbers such that $n, m \geq 2$.
Let $i$ be a natural number such that $1 \leq i \leq m$
Let $\vec{x} = (\vec{x_1},\vec{x_2},\ldots,\vec{x_{m-1}})$ be arbitrary points of $(S^n)^{m-1}$.
We define $\rho_{i,\vec{x}} : S^n \to (S^n)^m$ as follows.
\[
    \begin{array}{r@{\,\,}c@{\,\,}c@{\,\,}c}
       \rho_{i, \vec{x}}\colon&S^n&\xrightarrow{\qquad\qquad\qquad}&(S^n)^m\\
       &\rotatebox{90}{$\in$}&&\rotatebox{90}{$\in$}\\
       &p&\xmapsto{\qquad\qquad\qquad}&((\vec{x_i})_1,(\vec{x_i})_2,\ldots,(\vec{x_i})_{i-1},p,(\vec{x_i})_{i+1},\ldots,(\vec{x_i})_{m-1})
    \end{array}
\]
Let $\Delta \colon S^n \to (S^n)^m$ be a diagonal map.
In this case, we have the following relationship between maps induced on $n$ dimensional homology groups.
\[
\sum_{i=1}^{m}(\rho_{i,\vec{x}})_{\ast} = \Delta_{\ast}\].
\end{fact}

From now on, we fix an element of \textbf{Ch}$(\mathcal{A})$, $c_0$.
We set $\vec{c_0} = (c_0, c_0,\ldots,c_0)$.\\

From the Fact  \ref{sphere} and the Figure \ref{kakanzusiki_1} and \ref{kakanzusiki_2}, we have the following.


\begin{prop}\label{1}
Let $\mathcal{A} = \{H_1,H_2,\ldots,H_n\}$ be a real central indecomposable hyperplane arrangement in $V$ such that $n =|\mathcal{A}|\ge3$.
Let $m$ be a natural number.
Let $\Phi \colon \textbf{Ch}(\mathcal{A})^m \to \textbf{Ch}(\mathcal{A})$ be an admissible map.
In this case, we have 
\[
\sum_{i=1}^{m}\bigg(
\soejitilde{-0.2}{0.5}{-0.5}{1.3}{{\Phi} \circ \rho_{i, \vec{c_0}}}{M}
\bigg)_{\ast}
= \mathrm{id}\].
\end{prop}
\begin{proof}

\begin{align*}
\sum_{i=1}^{m}\bigg(
\Soejitilde{-0.2}{0.5}{-0.7}{1.3}{{\Phi} \circ \rho_{i, \vec{c_0}}}{M}
\bigg)_{\ast}
&=
\sum_{i=1}^{m}\bigg(
\Soejitilde{0}{0}{-0.2}{1.2}{\Phi}{M}
\hspace{0.6em}\circ
\Soejitilde{0}{0.5}{-0.2}{0.8}{\rho_{i, \vec{c_0}}}{M}
\hspace{0.6em}\bigg)_{\ast}\\
&=
\bigg(
\Soejitilde{0}{0}{-0.2}{1.2}{\Phi}{M}
\hspace{0.6em}\bigg)_{\ast}
\circ
\sum_{i=1}^{m}
\bigg(
\Soejitilde{0}{0.5}{-0.2}{0.8}{\rho_{i, \vec{c_0}}}{M}
\hspace{0.6em}\bigg)_{\ast}\\
&=
\bigg(
\Soejitilde{0}{0}{-0.2}{1.2}{\Phi}{M}
\hspace{0.6em}\bigg)_{\ast}
\circ
\bigg(
\Soejitilde{0}{0}{-0.2}{1.2}{\Delta}{M}
\hspace{0.6em}\bigg)_{\ast}\\
&=
\bigg(
\Soejitilde{0}{0}{-0.2}{1.2}{\Phi}{M}
\hspace{0.6em}\circ
\Soejitilde{0}{0}{-0.2}{1.2}{\Delta}{M}
\hspace{0.6em}\bigg)_{\ast}\\
&=
\bigg(
\Soejitilde{0}{0.5}{-0.2}{1.3}{\Phi\circ\Delta}{M}
\hspace{0.6em}\bigg)_{\ast}\\
&=
\bigg(
\Soejitilde{0}{0.5}{-0.2}{1.3}{\mathrm{id}}{M}
\hspace{0.6em}\bigg)_{\ast}\\
&=\mathrm{id}
\end{align*}

\end{proof}

In the following discussion, we show that $i_0$ of Proposition \ref{1} is actually a dictator.
This discussion is greatly inspired by Manabe~\cite{Mana}.


The following lemma is an important property of IIA.

\begin{lemm}[an important property of IIA]\label{2}
Let $\mathcal{A}$ be a real central hyperplane arrangement in $V$.
When a map $f\colon\textbf{Ch}(\mathcal{A}) \to \textbf{Ch}(\mathcal{A})$ satisfies IIA, the following five properties are equivalent.\\

\noindent
$(\mathrm{i})$ $f$ is one-to-one.\\
$(\mathrm{ii})$ For each element $H$ of $\mathcal{A}$, $\phi_{H}$ is $\mathrm{id}$ or $\bar{\mathrm{id}}$. \\
     $($Here, $\phi_{H}$ is defined in definition $\ref{IIA2}$, and $\bar{\mathrm{id}}$ is the map which exchanges the two elements of $\textbf{Ch}(\{H\})$.$)$\\
$(\mathrm{iii})$ $\soejitilde{0}{0.5}{-0.1}{1.2}{f}{M}$ is not null-homotopic.\\
$(\mathrm{iv})$  the degree of $\soejitilde{0}{0.5}{-0.1}{1.2}{f}{M}$ is $1$ or $-1$.\\
$(\mathrm{v})$  $\soejitilde{0}{0.5}{-0.1}{1.2}{f}{M}$ is a homeomorphism.\\
\end{lemm}

\begin{proof}
(v) $\Rightarrow$ (iv) $\Rightarrow$ (iii) $\Rightarrow$ (ii) is clear. 
(i) $\Rightarrow$ (v) is also obvious. 
To show (ii) $\Rightarrow$ (i), it is enough to derive the injectivity of $f$ from (ii), and this is easy.
\end{proof}

\begin{cor}\label{3}
Let us assume $g \colon \textbf{Ch}(\mathcal{A}) \to \textbf{Ch}(\mathcal{A})$ is one-to-one, satisfies IIA, and has a fixed point.
In this case, $g$ is $\mathrm{id}$.
\end{cor}
\begin{proof}
From the above proposition, for each $H$, $\phi_{H}$ must be either id or $\bar{\mathrm{id}}$. However, the condition that it has a fixed point implies that they must all be identity.
\end{proof}

From Proposition \ref{1}, Lemma \ref{2}, and Corollary \ref{3}, the following hold.
\begin{prop}
Let $\mathcal{A} = \{H_1, H_2, \ldots, H_n\}$ be a real central hyperplane arrangement in $V$ such that $n =|\mathcal{A}|\ge3$.
Let $m$ be a natural number.
Let $\Phi \colon \textbf{Ch}(\mathcal{A})^m \to \textbf{Ch}(\mathcal{A})$ be an admissible map.
In this case, there exists some $i_0 (1 \leq i_0 \leq m)$, such that $\Phi \circ \rho_{i_0,\vec{c_0}}$ is $\mathrm{id}$.
\end{prop}
\begin{proof}
From Proposition \ref{1}, there exists some $i_0$ $(1 \leq i_0 \leq m)$ such that $\Phi \circ \rho_{i_0,\vec{c_0}}$ is not null-homotopic.
In this case, from Proposition \ref{1}, $\Phi \circ \rho_{i_0,\vec{c_0}}$ is one-to-one.
$\Phi$ and $\rho_{i_0,\vec{c_0}}$ satisfy IIA, therefore $\Phi \circ \rho_{i_0,\vec{c_0}}$ also satisfies IIA.
Furthermore, $\Phi$ satisfies PAR, and $\vec{c_0} = (c_0,c_0,\ldots,c_0)$, $\Phi \circ \rho_{i_0,\vec{c_0}}$ has $c_0$ as a fixed point.
Therefore, from Corollary \ref{3}, $\Phi \circ \rho_{i_0,\vec{c_0}}$ is $\mathrm{id}$.
\end{proof}

Up to this point, it has been shown that for $\vec{c} \in \textbf{Ch}(\mathcal{A})^m$ where all components except the $i_0$ component are the same, $\Phi$ coincides with the projection onto the $i_0$-th component.
If this is also true for all $\vec{c} = (c_1,c_2,\ldots,c_{m-1}) \in \textbf{Ch}(\mathcal{A})^{m-1}$, then the proof is complete.

Let us $\textrm{IIA}_{\textrm{bij}}(\mathcal{A})$ be the set of all bijections from $\textbf{Ch}(\mathcal{A})$ to $\textbf{Ch}(\mathcal{A})$ that satisfy IIA.
For $f$ in $\textrm{IIA}_{\textrm{bij}}(\mathcal{A})$, we can think of a sequence of maps $(\phi_{H_1}, \phi_{H_2}, \ldots, \phi_{H_n})$. (Here, $\mathcal{A} = \{H_1, H_2, \ldots, H_n\}$).
From Proposition \ref{1}, this sequence consists of id or $\bar{\mathrm{id}}$.
We can make $\textrm{IIA}_{\textrm{bij}}(\mathcal{A})$ into a metric space by introducing Hamming distance of the sequences above.
Furthermore, for any natural number $k$, we can define a metric can be defined for any two elements $\vec{c_1}, \vec{c_2}$ in $\textbf{Ch}(\mathcal{A})^{k}$ by setting the metric to be the number of $H \in \mathcal{A}$ that separates $\{\vec{c_1}, \vec{c_2}\}$.

We now define the following map.
 \[
    \begin{array}{r@{\,\,}c@{\,\,}c@{\,\,}c}
       \phi\colon&\textbf{Ch}(\mathcal{A})^{m-1}&\longrightarrow&\textrm{IIA}_{\textrm{surj}}(\mathcal{A})\\
       &\rotatebox{90}{$\in$}&&\rotatebox{90}{$\in$}\\
       &\vec{c}&\longmapsto&\Phi \circ \rho_{i_0,\vec{c}}
    \end{array}
\]

By the preceding discussion, the domain and codomain of this map are metric spaces, and it is easy to see that $\phi$ satisfies $d(\phi(\vec{c_1},\vec{c_2})) \leq d(\vec{c_1},\vec{c_2})$ 
(In other words, $\phi$ is a non-expanding map.)

By the basic discussion of hyperplane arrangements, the following holds.

\begin{fact}\label{last}
Let $\mathcal{A} = \{H_1,H_2,\ldots,H_n\}$ be a real central hyperplane arrangement in \textit{V} such that $n =|\mathcal{A}|\ge3$.
Suppose there is no $\{H\} \in \mathcal{A}$ such that $\mathcal{A} = (\mathcal{A}-\{H\})\uplus  \{H\}$ $($particularly in the case that $\mathcal{A}$ is indecomposable).
In this case, for any element in 
$\mathrm{IIA_{surj}}(\mathcal{A})$ except for $\mathrm{id}$, the distance from $\mathrm{id}$ is at least $2$.
\end{fact}

By Proposition \ref{last}, for any $\vec{c} = (c_1,c_2,\ldots,c_{m-1}) \in \textbf{Ch}(\mathcal{A})^{m-1}$, it follows that $\Phi \circ \rho_{i_0,\vec{c}}$ is $\mathrm{id}$.
Therefore, the main theorem (Theorem \ref{mainthm}) is proven.

\section{Appendix}\label{hosoku1}
\begin{prop}
Let $k$ be a natural number.
For the inclusion map, $A_{\mathcal{A},k} \hookrightarrow U_{\mathcal{A},k}$ there is a deformation retract $r \colon U_{\mathcal{A},k} \rightarrow A_{\mathcal{A},k}$.
Especially, $A_{\mathcal{A},k} \hookrightarrow U_{\mathcal{A},k}$ is homotopy equivalence.
\end{prop}

\begin{proof}
We define a function $\rho \colon \mathbb{R} \rightarrow \mathbb{R}$ as follows.

\begin{equation}  \label{eq: cases f}
\rho(x)=
    \begin{cases}
        1   &   \text{$1<x$}  \\
        x        &   \text{$0<x\leq1$}  \\
        -x        &   \text{$-1\leq x<0$}  \\
        1   &   \text{$x<-1$}
    \end{cases}
\end{equation}

Using this, we define a map $H \colon U_{\mathcal{A},k}\times [0,1] \rightarrow U_{\mathcal{A},k}$ as follows.
\begin{equation}  \label{eq: cases f}
H(\vec{y}=(\vec{x_1},\vec{x_2},\dots,\vec{x_m}),t)=
    (1-t)\vec{y} + t\frac{1}{(\underset{H \in \mathcal{A}}{\textrm{max}}(\underset{1\leq i \leq m}{\textrm{min}}(\rho(\alpha_{H}(\vec{x_i})))))}\vec{y}
\end{equation}

This is the desired deformation retraction.

\end{proof}

\section*{Acknowledgement}
I would like to express my gratitude to Professor Mikio Furuta for his constant encouragement and advice.
I also express my gratitude to the senior students in the graduate office for their helpful discussions and encouragement.

\nocite{1194c2af-bbd9-34a4-ab4a-a33e7650b716}
\nocite{orlik1992arrangements}
\nocite{*}
\bibliography{ref.bib}
\bibliographystyle{plain}

\end{document}